\documentclass[11pt,twoside]{article}

\usepackage{mathrsfs,amsfonts,amsmath,amssymb}
\usepackage{latexsym}
\newtheorem{satz}{Theorem}

\newtheorem{theorem}[satz]{Theorem}
\newtheorem{conjecture}[satz]{Conjecture}
\newtheorem{lemma}[satz]{Lemma}

\newtheorem{corollary}[satz]{Corollary}

\def\sbeq{\subseteq}

\def\Z{\mathbb {Z}}
\def\F{\mathbb {F}}
\def\E{\mathsf{E}}

\def\P{{\cal P}}

\def\d{\delta}

\def\le{\leqslant}
\def\ge{\geqslant}
\def\x{{\overline x}}
\def\v{{\overline v}}
\def\y{{\overline y}}
\def\z{{\overline z}}
\def\_phi{\varphi}

\def\FF{\widehat}

\def\ov{\overline}

\def\f{{\mathbb F}}

\def\F{\mathbb{F}}

\def\Fp{\mathbb{\F}_p}

\newcommand{\abs}[1]{\left| {#1} \right|}

\author{Tomasz Schoen\footnote{The author is supported by National Science Centre, Poland grant 2019/35/B/ST1/00264.},    
Ilya D. Shkredov\footnote{The author is supported by the grant of the Russian Government N 075-15-2019-1926.}}
\title{\bf  Character sums estimates  and an application to a problem of Balog
}
\date{}
\begin{document}
	\maketitle

\begin{abstract}
	We prove   new bounds for sums of multiplicative characters over sums of set with small doubling and  applying this result we break the square--root barrier in a problem of Balog concerning products of differences in a  field of prime order.   
\end{abstract}

\bigskip  

\centerline{\sc 1. Introduction}
\bigskip

For a  prime number $p$ let  $\F_p$ be the prime field and let $\chi$ be a nontrivial multiplicative character modulo $p$.
We will deal with  a problem of estimating the exponential sum of the form 
\begin{equation}\label{f:def_sum}
\sum_{a\in A,\, b\in B}\chi(a+b) \,,
\end{equation}
where $A,B$ are arbitrary subsets of the field  $\F_p$.
Such exponential sums  were studied by numerous authors, see e.g. \cite{Chang}, \cite{DE1}, \cite{Hanson}, \cite{Kar}--\cite{Kar2}, \cite{SV}, \cite{Volostnov_E3}. 
One of the most important conjecture concerning character sums  is {\it the  Paley graph conjecture}, see  \cite{Chang}.
\begin{conjecture}
	For every  $\delta>0$  there is $\tau=\tau(\delta)>0$ such that for every prime number $p>p(\tau)$ and for  any set $A, B\sbeq \Fp$ with 
	$\abs{A}>p^{\delta}$ and $\abs{B}>p^\delta$ we have 
	\begin{equation}
	\Big|\sum_{a\in A,\, b\in B}\chi(a+b)\Big|<p^{-\tau}\abs{A}\abs{B} \,.
	\label{Paley}\end{equation}
\end{conjecture}
Currently there are  very few results regarding the above conjecture.
The only  affirmative answer
was obtained
 in the following case   
\begin{equation}\label{karacuba}
\abs{A}>p^{\frac12+\delta},~~\abs{B}>p^{\delta}\,,
\end{equation}
 see \cite{Kar}---\cite{Kar2}.
 M.--C. Chang \cite{Chang}
proved a result towards the conjecture
for  sets $A$ and $B$, where one  of them has a small sumset.
Her work was continued in \cite{Hanson}, \cite{SV}, \cite{Volostnov_E3} and  
currently the best result is the following theorem of Volostnov \cite{Volostnov_E3}.

\begin{theorem}\cite{Volostnov_E3}
\label{main_theorem_Volostnov}
	Let $A,\,B\subset\mathbb{F}_p$  and  $K,L,\delta > 0$ be such that
	\begin{equation*}
	|A|, |B| >p^{\frac{1}{3}+\delta}, \quad \quad \mbox{ and } \quad \quad 
	|A+A|< K|A|,~
	|B+B| < L|B| \,.
	\end{equation*}
	Then there is a positive function $C(K)$ such that for any nontrivial multiplicative character $\chi$ modulo $p$ one has
	\begin{equation}\label{f:main_theorem_Volostnov}
	\Big|\sum_{a\in A,\, b\in B} \chi(a+b)\Big| \ll p^{-\d^2 /C(K)} |A| |B| \,,
	\end{equation}
	provided that $p>p(\delta, K, L)$. 
\end{theorem}

The following theorem   is the main result of our paper.

\begin{theorem}
	\label{main_theorem}
	Let $A,\,B\subset\mathbb{F}_p$  and  $K,L,\delta > 0$ be such that $|A|>p^\d,\,|B| >p^{1/3+\d},$
	\begin{equation*}
	|A| |B|^2 >p^{1+\delta}   \quad \quad \mbox{ and } \quad \quad 
	|A+A|< K|A|,~
	|B+B| <L|B| \le p^{\d/2} |B| \,.
	\end{equation*}
	Then there is an absolute constant $c>0$ such that for any nontrivial multiplicative character $\chi$ modulo $p$ one has
	\begin{equation}\label{f:main_theorem}
	\Big|\sum_{a\in A,\, b\in B} \chi(a+b)\Big| \ll \exp (-c(\d^4 \log p /\log^2 K)^{1/3} )|A||B|   \,,
	\end{equation}
	provided  that
	$
		\log^5 K \ll \d^{4} \log p \,.
$
\end{theorem}	

 In the previous works \cite{Chang}, \cite{SV}, \cite{Volostnov_E3} the doubling $K$ can tend to infinity very slowly as the function $C(K)$ has an exponential nature.
However, our result is applicable for a much wider range of $K$. A  new structural result given in section 4, is what makes our aproach much more effective and it does not require us to use Freiman's theorem.

We apply Theorem \ref{main_theorem} to prove a new sum-product type  result. 
Let us recall a well--known Theorem of  Balog \cite[Theorem 1]{Balog} (see also \cite{Petridis}).

\begin{theorem}\cite{Balog}
	\label{t:Balog}
	Let $A\subseteq \F_q$, $q=p^s$ with $|A| \ge q^{1/2+1/2^k}$ for some positive integer $k$. 
	If $A$ is an additive subgroup of $\mathbb{F}_{q}$ assume additionally  that $|A| \geqslant q^{1/2}+1$. 
	Then  
	$$
		(A-A)^{2k+1}=\mathbb{F}_{q} .
	$$
\end{theorem}

It easy to see that  Balog's theorem does not hold for set of size smaller than $q^{1/2}$. 
Let $A$  to be a nontrivial subfield of $\F_{p^2}$ then $|A|= p =|\F_{p^2} |^{1/2}$ and any combinations of  sums and products of elements from  $A$ belong to $A$. 
However, one may hope for an improvement of the  Balog's result  for fields $\F_p$, when $p$ is a prime number.  
	We break the square--root barrier in Theorem \ref{t:Balog} for subsets of  $\F_p$ and  we only need a few operations to generate the whole field.

\begin{theorem}
	There is a positive constant $c$ such that for every  $A\subseteq \F_p$ with 
	$$|A| \gg \exp(- c\log^{1/5} p) p^{1/2}  $$ we have
	\begin{equation}\frac{2A-2A}{A-A}  = \F_p \text{~~~or~~~} 
	\left(\frac{A-A}{A-A} \right)^2 (A-A) = \F_p \,.
	\end{equation}
	\label{t:Balog_new}
\end{theorem}

It was proven in \cite{HP} that any two sets $A,B$ with all sums belonging to the set of quadratic residues satisfy $|A||B|\le \frac{p-1}{2} + |B\cap (-A)|$. 
It almost solves a well--known S\'{a}rk\"{o}zy's conjecture \cite{Sarkozy_R} on additive decompositions of the  quadratic residues except for the case when $A+B$ equals exactly  the set  of quadratic residues and when the sum is direct. 
It immediately follows from  Theorem \ref{main_theorem} (also see Theorem \ref{t4} below)  that such sets $A$ and $B$ cannot have small doubling. 



\bigskip 

\centerline{\sc 2. Notation}
\label{sec:definitions}

\bigskip 

In this section we collect  notation used in the paper. 
Throughout the paper by  $p$  we always  mean an odd prime number  and we put 
$\F_p = \Z/p\Z$ and $\F_p^* = \F_p \setminus \{0\}$. 
%
%
We denote the Fourier transform of a function  $f : \F_p \to \mathbb{C}$ by
\begin{equation}\label{F:Fourier}
\FF{f}(\xi) =  \sum_{x \in \F_p} f(x) e( -\xi \cdot x) \,,
\end{equation}
where $e(x) = e^{2\pi i x/p}$. 
The Plancherel formula states that 
\begin{equation}\label{F_Par}
\sum_{x\in \F_p} f(x) \ov{g (x)}
=
\frac{1}{p} \sum_{\xi \in \F_p} \widehat{f} (\xi) \ov{\widehat{g} (\xi)} \,.
\end{equation}
The convolution of functions  $f,g : \f_p \to \mathbb{C}$ is defined by
\begin{equation}\label{f:convolutions}
(f*g) (x) := \sum_{y\in \f_p} f(y) g(x-y) \,.
\end{equation}
Clearly, we have
\begin{equation}\label{f:F_svertka}
\FF{f*g} = \FF{f} \FF{g} \,.
\end{equation}
We use the same capital letter to denote  set $A\subseteq \F_p$ and   its characteristic function $A: \F_p \to \{0,1 \}$. 
For any two sets $A,B \subseteq \f_p$ the additive energy of $A$ and $B$ is defined by
$$
\E^{+} (A,B) = |\{ (a_1,a_2,b_1,b_2) \in A\times A \times B \times B ~:~ a_1+b_1 = a_2+b_2 \}| \,.
$$
If $A=B$, then  we simply write $\E^{+} (A)$ for $\E^{+} (A,A)$.
Combining \eqref{F_Par} and \eqref{f:F_svertka},we derive  
\begin{equation}\label{f:energy_Fourier}
\E^{+}(A,B) = \frac{1}{p} \sum_{\xi} |\FF{A} (\xi)|^2 |\FF{B} (\xi)|^2 \,.
\end{equation}
 The multiplicative energy $\E^{\times} (A,B)$  is defined in an analogous way and it can be  expressed similarly by applying 
the Fourier transform on  group $\F_p^*$ and the  multiplicative convolution.
Given any two sets $A,B\subset \F_p$, define the  \textit{sumset}, the \textit{product set} and the \textit{quotient set} of $A$ and $B$ as 
$$A+B:=\{a+b ~:~ a\in{A},\,b\in{B}\}\,,$$
$$AB:=\{ab ~:~ a\in{A},\,b\in{B}\}\,,$$
and 
$$A/B:=\{a/b ~:~ a\in{A},\,b\in{B},\,b\neq0\}\,,$$
respectively. We define $k$-fold sumsets and  product sets analogously, for example  $2A-2A=A+A-A-A$ and $(A-A)^2=(A-A)(A-A)$.

All logarithms are to base $2.$ The signs $\ll$ and $\gg$ are the usual Vinogradov symbols.
For a positive integer $n,$ we put $[n]=\{1,\ldots,n\}.$ 

\bigskip 

\centerline{\sc 3. Auxiliary results} 
\bigskip 

Here we list    results that we will use over the course of  Theorem \ref{main_theorem} proof. 
The Lemma was established in  \cite[Lemma 14]{SV} and it is a corollary to  Weil's bound.

\begin{lemma}\label{l:davenport}
	Let $\chi$ be a nontrivial multiplicative character, $I$ be a discrete interval in $\F_p$  and $r$ be a positive integer.
	Then 
	\begin{eqnarray*}
	\sum_{u_1,u_2\in \mathbb{F}_p}\Big |\sum_{t\in I}\chi(u_1+t)\overline{\chi}(u_2+t)\Big |^{2r}< p^2 \abs{I}^r r^{2r}+4r^2p\abs{I}^{2r} \,.
	\end{eqnarray*}
\end{lemma}
We will also need a corollary to   points/planes incidences theorem of Rudnev \cite{Rudnev_pp}. 

\begin{theorem}
	Let $p$ be an odd prime, $\mathcal{P} \subseteq \F_p^3$ be a set of points and $\Pi$ be a collection of planes in $\F_p^3$. 
	Suppose that $|\mathcal{P}| \le |\Pi|$ and that $k$ is the maximum number of collinear points in $\mathcal{P}$. 
	Then the number of point--planes incidences satisfies 
	\begin{equation}\label{f:Misha+_a}
	|\{ (q,\pi) \in \mathcal{P} \times \Pi ~:~ q\in \pi \}|  - \frac{|\mathcal{P}| |\Pi|}{p} \ll |\mathcal{P}|^{1/2} |\Pi| + k |\Pi| \,.	
	\end{equation}
	\label{t:Misha+}	
\end{theorem}

\begin{corollary}\label{l:E_3}
	Let $A,B \subset \F_p$ be such that  $|A|,|B| <\sqrt{p}$ and $|A+A| \le K|A|$, $|B+B| \le L|B|$. 
	Then the system
\begin{equation}\begin{cases}
\,\frac{b_1}a=\frac{b_1'}{a'}\\
\,\frac{b_2}a=\frac{b_2'}{a'}
\end{cases}
\label{system}
\end{equation}
	has 
\begin{equation}\label{f:bound_E^t_3}
O(K^{5/4} L^{5/2} |A||B|^2 \log p + |A|^2 |B|)
\end{equation}
	solutions in $(a,a',b_1,b_1',b_2,b_2')\in A^2\times B^4$. Furthermore,
for an arbitrary set $B$ with  $K\le |B|$
there are  at most 
\begin{equation}\label{f:bound_E^t_3-1}
O(K^{3/2} |A| |B|^{5/2})
\end{equation}
solutions. 
\end{corollary}
\begin{proof} 
	The first part of the corollary is \cite[Lemma 2]{Volostnov_E3}, so it is enough to prove \eqref{f:bound_E^t_3-1}. 
	Denote by $\sigma$  the number of the solutions to \eqref{system}. 
	Clearly,  we have $\sigma \le |B| \E^{\times} (A,B) \le |B|^2 |A|^2$.
	Thus, we can assume that $|B| \le |A|^2/K^3,$ as otherwise $|B|^2 |A|^2\le K^{3/2} |A| |B|^{5/2}$ and the assertion follows. 
	Since every $x\in A$ can be written in $|A|$ ways as $x=s-a$, $s\in A+A$ and  $a\in A$ it follows that 
	\[
	\E^{\times} (A,B) \le |A|^{-2} |\{(s-a)b = (s'-a')b' ~:~ s,s'\in A+A,\, a,a'\in A,\, b,b'\in B \}|\,.
\]  
	Next, we apply Theorem \ref{t:Misha+} with 
	$$\mathcal{\P}=\{(s,b',a'b'):  a'\in A, b'\in B, s\in A+A \}$$
	and 
	$$\Pi = \big \{(x-a)b = s' y - z: a\in A, b\in B, s'\in A+A \big\}\,.$$ 
	If all elements of  $A, B$ and $A+A$ are nonzero then $|\Pi|=|A||B||A+A|,$ otherwise one can remove zero from those sets.
	Observe that every incidence between a plane $(x-a)b = s' y - z$ from $\Pi$  and a point $(s,b',a'b')$ from $\mathcal{\P}$ gives a solution to
	\begin{equation}\label{s}
	(s-a)b = (s' - a')b'\,
	\end{equation}
	and each solution to \eqref{s} provides a point/plane incidence. 
	 One can easily check
that the maximal number of collinear points (points belonging a
line) in $\mathcal{P}$ is just the maximal size of  "skew Cartesian
product" $(s,b',a'b')$, so $ k = \max\{|A|, |B|, |A+A| \} =
\max\{|B|, |A+A|\}.$
	By Theorem \ref{t:Misha+}, in view of $K\le |B| \le |A|^2/K^3$ and $|A|<\sqrt{p}$, 
	we have
\begin{eqnarray*}
	\sigma 
	&\ll&
	|B| |A|^{-2} \left( \frac{|A+A|^2 |A|^2 |B|^2}{p} + (|A||A+A||B|)^{3/2} + |A||B||A+A|(|B| + |A+A|)\right) \\
	&\ll&
		\frac{|A+A|^2 |B|^3}{p} + K^{3/2} |A| |B|^{5/2} \ll K^{3/2} |A| |B|^{5/2} \,.
\end{eqnarray*}
	This completes the proof.
$\hfill\Box$
\end{proof}

\bigskip 

The last two results \cite{Redei} (see also \cite{Szonyi}) and \cite{bsg} will be used in the proof of Theorem \ref{t:Balog_new}. 
A modern form of Lemma \ref{l:Redei} can be found in \cite{BSW_Redei}.

\begin{lemma}
	Let $A\subseteq \F_p$ be a set with $|A| > 1$. 
	Then
\[
	\left|\frac{A-A}{A-A} \right| \ge \min \left\{ p,\, \frac{|A|^2+3}{2} \right\} \,.
\]
\label{l:Redei}
\end{lemma}

\begin{lemma}
Let $A$ be a subset of an abelian group such that $\E^+(A)=|A|^3/K.$ Then there exists $A'\sbeq A$ such that $|A'|\gg  |A|/K$ and
$$|A'-A'|\ll K^{4}|A'|\,.$$
\label{l:bsg}
\end{lemma}

\bigskip
\centerline{\sc 4. Multiplicative structures in sumsets}
\bigskip

The next lemma directly follows from Proposition 4.1 in \cite{Sanders_2A-2A}.

\begin{lemma}
	Suppose that $G = (G,+)$ is a group  and $k \in \mathbb{N}.$ Let $A\subset G$ be a finite non--empty set such that $|A + A| \leq K|A|.$ 
	Then there is  a  set  $X\sbeq A-A$  of size  
	\begin{equation}\label{f:Sanders_lower_bound}
	|X| \geqslant \exp \left(-O(k^{2} \log^2 K) \right)|A|
	\end{equation}
	such that
	$kX\sbeq 2A-2A.$
\label{l:Sanders}
\end{lemma}

\bigskip

\begin{lemma}\label{l:bukh}
	Let  $A$ be a finite subset of  an abelian group such that $|A+A|\le K|A|$. Then  for any  $\lambda_{i} \in \mathbb{Z} \backslash\{0\}$ we have
	$$
	\left|\lambda_{1} \cdot A+\cdots+\lambda_{k} \cdot A\right| \leqslant K^{O\left(\Sigma_{i} \log \left(1+\left|\lambda_{i}\right|\right)\right)}|A| \,.
	$$
\end{lemma}
 
The next result is an important ingredient of our approach,  which shows that additively rich sets  contain product of surprisingly large sets. 
	Our argument is based on  \cite{CRS}.

\begin{lemma}
	Let $A\subseteq \F_p$ be a set such that  $|A+A| \le K|A|$.
	There exists an absolute constant $C >0$
	 such that for any positive integers  $d\ge 2$ and $l$ there is 
	a set $Z$
of size  $|Z| \ge   \exp \left(-C l^{3} d^2 \log^2 K \right) |A|$ with
\begin{equation}\label{f:mult_inclusion_3}	
[d^l] \cdot Z \subseteq 2A-2A \,.
\end{equation}
\label{l:mult_inclusion}
\end{lemma}
\begin{proof} 
 Let $l$ and $d\ge 2$ be positive integers
and  $X$ be a set retrieved by applying Lemma \ref{l:Sanders} to  set $A$ and  $k=2(d(l-1)+1)$, so $|X|\ge \exp{(-Cl^2d^2
\log^2 K)}|A|$. 
	We define a  map 
	$$f: X^{l+1} \to  (X+X) \times (X+d\cdot X) \times \dots \times (X+d^{l-1}\cdot X)$$
 by  
	$$
	 f(\x) = (x_1+x_0,x_2+dx_0, \dots, x_{l}+d^{l-1}x_0) \,,
	$$
	where $\x=(x_0,x_1, \dots, x_{l}).$
	Let $\mathcal{X}$ be  the image of $f$ then by Lemma  \ref{l:bukh} and the Pl\"unnecke inequality \cite{TV}  we have
	\begin{eqnarray*}
	|\mathcal{X}|&\le& |X+X||X+d\cdot X|\cdots|X+d^{l-1}\cdot X|\\
	&\le &|(A-A)+(A-A)||(A-A)+d\cdot (A-A)|\cdots|(A-A)+d^{l-1}\cdot (A-A)|\\
	&\le& \exp{(O(l^2\log d\log  K))}|A-A|^l\\
	&=&\exp{(O(l^2\log d\log K))}|A|^l\,.
	\end{eqnarray*}
	For a vector ${\z}\in \mathcal{X}$ put 
	$$r({\z})=|\{{\x}\in X^{l+1}: f({\x})={\z}\}| \,.$$
	Clearly, 
	$$\sum_{{\z}\in \mathcal{X}} r({\z}) = |X|^{l+1}\,,$$
	so there is a ${\v}$ such that 
	\begin{eqnarray*}
	r({\v}) \ge |X|^{l+1}/|\mathcal{X}|&\ge& \exp{(-O(l^2\log d\log K))}|X|^{l+1}|A|^{-l}\\
	&\ge& \exp{(-O(l^3d^2\log^2 K))}|A|\,.
	\end{eqnarray*}
	Put ${\mathbf Y}=\{\x\in X^{l+1}: f(\x)={\v}\}$ and let $\x\in {\mathbf Y}$ be a fixed element.
	Then for any ${\y}\in {\mathbf Y}$  and any $j\in [0,l-1]$  we have  
	\[
		 d^j(x_0-y_0) = x_j - y_j  \in X-X \,.
	\]
	Observe that for distinct  vectors ${\y},{\y'}\in {\mathbf Y}$ the  corresponding coefficients $y_0$ and $y_0'$ must  also be distinct and hence  elements $x_0-y_0$ with $\y\in {\mathbf Y}$ form a set of size at least $\exp{(-O(l^3d^2\log^2 K))}|A|$. Denote this set by $Z$, then $d^j\cdot Z\sbeq X-X$ for $j\in [0,l-1]$, so
	$$[d^l] \cdot Z \subseteq (l(d-1) + 1) (X-X) = 2 (l(d-1)+1) X\sbeq 2A-2A\,,$$
which concludes the proof.
$\hfill\Box$
\end{proof}

\bigskip 

\centerline{\sc 4. Proof of the main results}

\bigskip 

Now we are ready to prove Theorem \ref{main_theorem}.

\bigskip

\begin{proof} 
Let $l$ and  $r\ge 2$ be parameters that will be specified later. 
By applying Lemma \ref{l:mult_inclusion}  to  the set $A,\, l$ and $d=2$ there is 
a set $Y$ with 
$|A| \ge |Y| \ge   \exp \left(-C l^{3} \log^2 K \right) |A|$
such that  $I\cdot Y\sbeq 2A-2A$, where $I:=[2^l]$. 
For fixed $x\in  I$, $y\in Y$ we have 
\begin{equation}\label{tmp:10.09_1-}
\Big|\sum_{a \in A, b \in B} \chi(a+b)\Big| \leqslant \sum_{a \in A}\Big|\sum_{b \in B} \chi(a+b)\Big|
= \sum_{a \in A-x y}\Big|\sum_{b \in B} \chi(a+b+x y)\Big| \,,
\end{equation}
so
\begin{equation}\label{tmp:10.09_1}
	\Big|\sum_{a \in A, b \in B} \chi(a+b)\Big|^{4r} \le (|I||Y|)^{-4r}\Big( \sum_{a \in 3A-2A}\, \sum_{x\in I,\, y\in Y} \Big|\sum_{b \in B} \chi(a+b+x y)\Big| \Big)^{4r} \,.
\end{equation}
For $a\in 3A-2A$ we put $B_a := B+a$ and 
$$
	\sigma:= \sigma_a = \sum_{x \in I, y \in Y}\Big|\sum_{b \in B} \chi(a+b+x y)\Big|=\sum_{x \in I, y \in Y}\Big|\sum_{b \in B_{a}} \chi(b+x y)\Big| \,.
$$ 
After a pplying the Cauchy--Schwarz inequality we obtain 
\begin{eqnarray*}
	\sigma^2 &=& \Big( \sum_{x \in I, y \in Y} \Big|\sum_{b \in B_{a}} \chi(b+x y) \Big| \Big)^2\\
	&\le&
		|I| |Y| \Big( \sum_{b,b' \in B_a}\,  \sum_{x \in I, y \in Y} \chi(b+xy) \overline{\chi} (b'+xy) \Big) \\
		&=&
		|I| |Y| \sum_{u_1,u_2} \nu(u_1,u_2)  \sum_{x \in I} \chi(u_1+x) \overline{\chi} (u_2+x) \,,
\end{eqnarray*}
where 
$$\nu(u_1,u_2) = |\{ (b,b',z) \in B^2_a \times Y ~:~ b/y = u_1, b'/y = u_2 \}|.$$ 
From the H\"{o}lder inequality we have
\[
	\sigma^{4r}	\le (|I||Y|)^{2r} \Big(\sum_{u_1,u_2} \nu(u_1,u_2) \Big)^{2r-2} \Big( \sum_{u_1,u_2} \nu^2 (u_1,u_2) \Big)
	\Big(\sum_{u_1,u_2} \Big|\sum_{x \in I} \chi(u_1+x) \overline{\chi} (u_2+x) \Big|^{2r} \Big) \,.
\]
Since $\sum_{u_1,u_2} \nu(u_1,u_2) =|B|^2|Y|$ and $\sum_{u_1,u_2} \nu^2 (u_1,u_2)$ is the number of solutions to \eqref{system} it follows from  Lemma \ref{l:davenport} and Lemma \ref{l:E_3} that
\begin{equation}\label{tmp:17.09_1}
	\sigma^{4r}	\ll (|I||Y|)^{2r} (|B|^2 |Y|)^{2r-2} (K_Y^{5/4} L^{5/2} |Y||B|^2 \log p + |Y|^2 |B|) (p^2 |I|^r r^{2r}+4r^2p |I|^{2r}) \,,
\end{equation}
where $K_Y = |Y+Y|/|Y| \le K|A|/|Y|$. We assume  that  
$$K_Y^{5/4} L^{5/2} |Y||B|^2 \log p \ge |Y|^2 |B|\,.$$
Furthermore, suppose that $l$ satisfies the inequality
$|I|^r = 2^{lr} \ge r^{2r} p,$
then
 \[
	\sigma^{4r} \ll 
	r^2
	p |I|^{2r} (|I||Y|)^{2r} (|B|^2 |Y|)^{2r-2}  L^{5/2} K^{5/4} |A|^{5/4} |B|^2 |Y|^{-1/4} \log p \,.
\]
Now we go back  to bound  \eqref{tmp:10.09_1}. By the Pl\"unnecke inequality \cite{TV} we have $|3A-2A|\le K^{5} |A|$, hence
\begin{eqnarray*}
	\Big|\sum_{a \in A, b \in B} \chi(a+b)\Big|^{4r} 
		&\ll& 
			K^{5r} |A|^{4r} r^{2} p |I|^{2r} (|I||Y|)^{-2r} (|B|^2 |Y|)^{2r-2}   L^{5/2} K^{5/4} |A|^{5/4} |B|^2 |Y|^{-1/4} \log p \\
&	= &
	r^2
	(|A| |B|)^{4r} \cdot \Big( \frac{K^{5r+5/4} L^{5/4}  |A|^{5/4} p\log p}{|Y|^{9/4} |B|^2}\Big) \,.
\end{eqnarray*}
	Recall that 
	$|Y| \ge   \exp \left(-C l^{3} d^2 \log^2 K \right) |A|$, so
\[
\Big|\sum_{a \in A, b \in B} \chi(a+b)\Big|^{4r} 
\ll
	r^2
	(|A| |B|)^{4r} \cdot \Big( \frac{K^{C_1 (l^{3} \log K + r)} L^{5/4} p\log p}{|A| |B|^2}\Big) 
\]
	for an absolute constant $C_1>0.$  Now we choose parameters for the above inequality. 
	By the assumption $L<p^{\d/2}$ and $|A||B|^2>p^{1+\d}$ hence  $L^{5/4} p(\log p) (|A| |B|^2)^{-1} < p^{-\d/4}$.
	We put
	$$l =\Theta( (\d \log p /\log^2 K)^{1/3})$$
	and 
$$r =\Theta( (\log p)/l)=\Theta(  \d^{-1/3} \log^{2/3} K \cdot \log^{2/3} p)$$
such that the inequalities $K^{C_1 l^3 \log K} < p^{\d/16}$ and $2^{lr} \ge r^{2r} p$ are satisfied. 
	Furthermore, by the assumption it follows that   $K^{5r} < p^{\d/16}$.
	Thus, we have 
\begin{equation}\label{tmp:17.09_2}
	\Big|\sum_{a \in A, b \in B} \chi(a+b)\Big|
		\ll 
		 2^{-c\d l /64}	|A||B|  
			\ll  
		\exp (-c'(\d^4 \log p /\log^2 K)^{1/3} )	|A||B| 
\end{equation}
for some absolute constants $c,c'>0.$

To complete the proof assume now that  $|Y|^2 |B|>K_Y^{5/4} L^{5/2} |Y||B|^2 \log p$. Then the same argument leads to the estimate
\[
\Big|\sum_{a \in A, b \in B} \chi(a+b)\Big|^{4r} 
\ll
	r^2
	(|A| |B|)^{4r} \cdot \Big( \frac{K^{5r}p}{|B|^3}\Big)\,,
\]
hence the required inequality holds provided that $|B| >p^{1/3+\d}$,
which completes the proof.
$\hfill\Box$
\end{proof}

\bigskip 

 Using the same argument and the estimate \eqref{f:bound_E^t_3-1} one can avoid any constraints for  doubling   and  size of $B$.

\begin{theorem}\label{t4}
	\label{f:bound_E^t_3'}
	Let $A,\,B\subset\mathbb{F}_p$  and $K,\delta > 0$ be such that  $|A| > p^\delta,$ 
	\begin{equation*}
	|A|^2 |B|^3 >p^{2+\delta}, \quad \quad \mbox{ and } \quad \quad 
	|A+A|< K|A| \,.
	\end{equation*}
	Then for any nontrivial multiplicative character $\chi$ modulo $p$ one has
	\begin{equation}\label{f:main_theorem_B}
	\Big|\sum_{a\in A,\, b\in B} \chi(a+b)\Big| \ll_\d\exp (-c(\d^4 \log p /\log^2 K)^{1/3} )|A||B| \,,
	\end{equation}
	provided that $\log^5 K \ll \d^{4} \log p \,.$
\end{theorem}	
\begin{proof} If $|B|>p^{1/2+\delta/5}$ then we use   Karacuba's result  \eqref{karacuba} hence we may  assume that $|B|\le p^{1/2+\delta/5}$. Then $|A|>p^{1/4+\delta/5}$ so
$K\le |B|$ 
and we can apply \eqref{f:bound_E^t_3'}.  Closely following the proof of Theorem \ref{main_theorem} and using \eqref{f:bound_E^t_3'} we obtain the estimate
\begin{eqnarray*}
	\Big|\sum_{a \in A, b \in B} \chi(a+b)\Big|^{4r} 
		&\ll& 
			K^{5r} |A|^{4r} r^{2} p |I|^{2r} (|I||Y|)^{-2r} (|B|^2 |Y|)^{2r-2}  K^{3/2} |A|^{3/2} |B|^{5/2} |Y|^{-1/2} \log p \\
&	= &5
	r^2
	(|A| |B|)^{4r} \cdot \Big( \frac{K^{C(l^3\log K+r)}   p\log p}{|A| |B|^{3/2}}\Big) \,.
\end{eqnarray*}
The  choice of parameters  $l,r$ from Theorem \ref{main_theorem} provides the required upper bound. $\hfill\Box$
\end{proof}



\bigskip

\centerline{\sc Proof of  Theorem \ref{f:bound_E^t_3'} }

\bigskip

\begin{proof}
	Put $|A| = \frac{\sqrt{p}}{M}$, where  $1\le M\le \exp(c\log^{1/5} p),$ and let
	$Q = \frac{A-A}{A-A}$.
	By Lemma \ref{l:Redei} we have $|Q| \ge p/(2M^2)$.   
	Suppose  that $\frac{2A-2A}{A-A} \not= \F_p$ and let
	  $\xi \notin \frac{2A-2A}{A-A}$. Equivalently, if 
		\begin{equation}\label{eq}
		\xi(a_1-a_2)=a_3+a_4-a_5-a_6\,,
		\end{equation}
		$a_i\in A$ then $a_1=a_2$ and
	$a_3+a_4-a_5-a_6=0,$ so there are 	$|A| \E^{+} (A) $ solutions to this equation.
	In terms of Fourier transform the number of solutions to \eqref{eq} can be written as 
\[
	|A| \E^{+} (A) = p^{-1} \sum_x |\FF{A}(\xi x)|^2 |\FF{A}(x)|^4 \ge \frac{|A|^6}{p} = \frac{|A|^4}{M^2} \,.
\]
	Hence 
	$$\E^{+} (A) \ge |A|^3/M^2$$
	and by Lemma \ref{l:bsg} there is $A'\subseteq A$ such that $|A'|\gg |A|/M$ and $|A'+A'|\ll M^4|A'|$. 
	Using multiplicative Fourier coefficients,  the number of representations of any $z\in \F^*_p$ in the form $q/q'(a-b)$ with $q,q'\in Q$ and $a,b\in A$ equals 
$$\frac{1}{p-1} \sum_{\chi} |\sum_{x\in Q} \chi (x)|^2 \Big( \sum_{a,b \in A'} \chi (z(a-b)) \Big)\,. $$
By Theorem \ref{main_theorem} there are positive constants $c,C$ such that  the above quantity can be bounded from below by
$$(2p)^{-1} |Q|^2 |A'|^2  - C\exp (-c(\log p /\log^2 M)^{1/3} )|Q| |A'|^2\,,$$
which is greater than
\begin{equation}\label{tmp:14.09_1} 
(4M^2)^{-1}|Q||A'|^2 - C\exp (-c(\log p /\log^2 M)^{1/3} )|Q| |A'|^2 \,,
\end{equation}  
for a positive constant $c.$
	If  \eqref{tmp:14.09_1} is positive for every $z$, then  $QQ(A-A)=\F_p$. 
However, our assumption on the size of $A$ implies that
\[
	C\exp (c(\log p /\log^2 M)^{1/3} ) > 4M^2\,,
\]  
which concludes the proof.
$\hfill\Box$
\end{proof}

\begin{corollary}
	There is a positive constant $c$ such that for every set $A\subseteq \F_p$ with $|A| \gg \exp(- c\log^{1/5} p) p^{1/2} $ we have 
	\[
	\frac{(2A-2A)^3}{(2A-2A)^2} = \F_p  \,.
	\]
\label{c:Balog_new}
\end{corollary}

\bigskip

\noindent {Faculty of Mathematics and Computer Science,\\ Adam Mickiewicz
University,\\ Umul\-towska 87, 61-614 Pozna\'n, Poland\\} {\tt
schoen@amu.edu.pl}

\bigskip

\noindent{Steklov Mathematical Institute,\\
ul. Gubkina, 8, Moscow, Russia, 119991}
\\
and
\\
IITP RAS,  \\
Bolshoy Karetny per. 19, Moscow, Russia, 127994\\
and 
\\
MIPT, \\ 
Institutskii per. 9, Dolgoprudnii, Russia, 141701\\
{\tt ilya.shkredov@gmail.com}

\end{document}